\documentclass[12pt]{amsart}
\usepackage{amssymb}
\usepackage{mathtools}
\usepackage[english]{babel}
\usepackage{epsfig}
\usepackage{mathptmx}
\usepackage{amsmath,amsfonts,amssymb}
\usepackage{mathtools}
\usepackage{mathrsfs}
\usepackage[all]{xy}
\usepackage{graphicx}
\usepackage{latexsym}
\usepackage{verbatim}
\numberwithin{equation}{section}

\def\Re{{\sf Re}\,}

\newcommand{\D}{\mathbb D}

\newcommand{\R}{\mathbb R}

\newcommand{\C}{\mathbb C}

\newcommand{\B}{\mathbb B}

\newcommand{\oD}{\overline{\mathbb D}}

\newcommand{\N}{\mathbb N}

\def\Re{{\sf Re}\,}

\def\Re{{\sf Re}\,}

\def\Re{{\sf Re}\,}

\def\1#1{\overline{#1}}
\def\2#1{\widetilde{#1}}
\def\3#1{\widehat{#1}}
\def\4#1{\mathbb{#1}}
\def\5#1{\frak{#1}}
\def\6#1{{\mathcal{#1}}}

\def\Re{{\sf Re}\,}

\newcommand{\mcite}[1]{\csname b@#1\endcsname}

\theoremstyle{theorem}

\def\Re{{\sf Re}\,}

\newtheorem{theorem}{Theorem}[section]
\newtheorem{lemma}[theorem]{Lemma}
\newtheorem{proposition}[theorem]{Proposition}
\newtheorem{corollary}[theorem]{Corollary}

\theoremstyle{definition}
\newtheorem{definition}[theorem]{Definition}

\theoremstyle{remark}
\newtheorem{remark}[theorem]{Remark}

\numberwithin{equation}{section}

\title[Rigidity of holomorphic maps]{Slice rigidity property of holomorphic maps Kobayashi-isometrically preserving  complex geodesics}

\author[F. Bracci]{Filippo Bracci$^\dag$}
\address{F. Bracci: Dipartimento di Matematica, Universit\`a di Roma ``Tor Vergata", Via della Ricerca
Scientifica 1, 00133, Roma, Italia.} \email{fbracci@mat.uniroma2.it}

\author[\L. Kosi\'nski]{\L ukasz Kosi\'nski$^{\dag\dag}$}
\address{\L. Kosi\'nski: Institute of Mathematics, Faculty of Mathematics and Computer Science, Jagiellonian University, \L ojasiewicza 6, 30-348 Krak\'ow, Poland} \email{lukasz.kosinski@uj.edu.pl}

\author[W. Zwonek]{W\l odzimierz Zwonek$^{\dag\dag\dag}$}
\address{W. Zwonek: Institute of Mathematics, Faculty of Mathematics and Computer Science, Jagiellonian University, \L ojasiewicza 6, 30-348 Krak\'ow, Poland} \email{wlodzimierz.zwonek@uj.edu.pl}


\thanks{$^\dag$Partially supported by PRIN {\sl Real and Complex
Manifolds: Topology, Geometry and holomorphic dynamics} n.2017JZ2SW5, by INdAM and  by the MIUR Excellence Department Project awarded to the
Department of Mathematics, University of Rome Tor Vergata, CUP E83C18000100006}
\thanks{$^{\dag\dag}$Partially supported by NCN grant SONATA BIS no.  2017/26/E/ST1/00723 of the National Science Centre, Poland}
\thanks{$^{\dag\dag\dag}$ Partially supported by the OPUS grant no. 2015/17/B/ST1/00996 of the National Science Centre, Poland}

\long\def\REM#1{\relax}

\begin{document}
\maketitle
\tableofcontents

\selectlanguage{english}
\begin{abstract}
In this paper we study the following ``slice rigidity property'': given two Kobayashi complete hyperbolic manifolds $M, N$ and a collection of complex geodesics $\mathcal F$ of $M$, when is it true that every holomorphic map $F:M\to N$ which maps isometrically every complex geodesic of $\mathcal F$ onto a complex geodesic of $N$ is a biholomorphism? Among other things, we prove that this is the case if $M, N$ are smooth bounded strictly (linearly) convex domains, every element of $\mathcal F$  contains a given point of $\overline{M}$ and $\mathcal F$ spans all of $M$. More general results are provided in dimension $2$ and for the unit ball.
\end{abstract}

\section{Introduction}

Let $M, N$ be complete (Kobayashi) hyperbolic complex manifolds and let $F:M\to N$ be a biholomorphism. Since $F$ is an isometry between the Kobayashi metric $k_M$ of $M$ and $k_N$ of $N$ and the Kobayashi distance $K_M$ of $M$ and $K_N$ of $N$, it follows that $k_M=F^\ast k_N$ and $K_M=F^\ast K_N$. In particular, if $\varphi:\D \to M$ is a complex geodesic, {\sl i.e.} a holomorphic map which is an isometry between the hyperbolic distance in the unit disc $\D:=\{\zeta\in \C: |\zeta|<1\}$ and the Kobayashi distance in $M$, it follows that $F\circ \varphi$ is a complex geodesic in $N$. 

The aim of this paper is to understand to what extent  the converse of the previous implication holds. More precisely, let $\mathcal F$ be a collection of complex geodesics $\varphi:\D \to M$. The collection $\mathcal F$ is said to be {\sl complete} if for every $z\in M$ there exists $\varphi\in \mathcal F$ such that $z\in \varphi(\D)$.

\begin{definition}
We say that the triple $(M,N,\mathcal F)$ satisfies the {\sl slice rigidity property} provided for any holomorphic map $F:M\to N$ such that $F\circ \varphi$ is a complex geodesic of $N$ for all $\varphi\in\mathcal F$ it follows that $F$ is a biholomorphism.
\end{definition}

The problem is trivial in the planar case so we assume in the whole paper that $n>1$.

Clearly, if $\mathcal F$ is too small (for instance just one complex geodesic) the slice rigidity property does not hold. However, even in case $\mathcal F$ is complete, the property in general does not hold:  in the bi-disc $\D\times \D$ the collection $\mathcal F$ of complex geodesics $\varphi_{w_0}:\D \to \D\times \D$ given by $\varphi_{w_0}(\zeta)= (\zeta, w_0)$, for $w_0\in \D$, is complete and the map $F:\D\times \D\to \D\times \D$ given by $F(z,w)=(z,0)$ maps isometrically all elements of $\mathcal F$ onto a complex geodesic of $\D\times \D$, but it is not a biholomorphism.

If $D\subset \C^n$ is a smooth bounded strictly (linearly) convex domain, it is known by Lempert \cite{Lem 1981, Lem 1982, Lem 1984} that for any $p\in \partial D$ and any $z\in D$ there exists a complex geodesic $\varphi:\D \to D$ (essentially unique up to pre-composition with M\"obius transformations), which extends smoothly up to $\partial \D$ such that $p, z\in \overline{\varphi(\D)}$.  The main result of this note is the following:

\begin{theorem}\label{main-intro}
Let $D, G\subset \C^n$ be smooth bounded strongly (linearly) convex domains and let $p\in \overline{D}$. Let $\mathcal F$ be a complete collection of complex geodesics of $D$ such that $p\in \overline{\varphi(\D)}$ for all $\varphi\in \mathcal F$. Then $(D,G,\mathcal F)$ satisfies the slice rigidity property.
\end{theorem}

Completeness of the collection of geodesics $\mathcal F$ is a reasonable hypothesis which is normally verified in practice. However, in the case of the unit ball $\B^n:=\{z\in \C^n: |z|<1\}$ we can drop such a condition. Recall that the image of every complex geodesic in $\B^n$ is the intersect of a complex affine line with $\B^n$:

\begin{theorem}\label{main-ball}
Let $q\in \C^n$. Let $\mathcal F$ be a  collection of complex geodesics of $\B^n$ such that every $\varphi\in \mathcal F$ is contained in a complex line passing through $q$ (in particular, if $q\in \overline{\B^n}$ this means that $q\in \overline{\varphi(\D)}$ for every $\varphi\in\mathcal F$). If there exists a non-empty open set $V\subset\B^n$ such that $V\subset \bigcup_{\varphi\in\mathcal F} \varphi(\D)$  then $(\B^n,G,\mathcal F)$ satisfies the slice rigidity property for any bounded strongly (linearly) convex domain $G\subset \C^n$ be   with real analytic boundary.
\end{theorem}
Note that, though the last theorem implies that the slice rigidity property might hold for collection $\mathcal F$ being not complete, the function $F_{1,1}:\B^2\to\D$ defined in \cite{Kos-Zwo 2018}, as an extremal solution of a degenerate $3$-point Nevanlinna-Pick problem in the ball, allows to construct a holomorphic mapping $F:=(F_{1,1},0):\B^2\to\B^2$ that acts as an isometry on a one-parameter family of complex geodesics $\{f_t\}_{t\in\R}$ (as defined in \cite{Kos-Zwo 2020}) but it is not an automorphism of $\B^2$ (see also \cite{Kos-Zwo 2016}). Consequently, this shows natural restrictions in a possible generalization of the last theorem that would aim to replace the openness of $V$ with some information on its dimension.

In dimension $2$ we  prove the slice rigidity property for (strongly) linearly convex domains and any complete collection $\mathcal F$, regardless the presence of a common fixed point for the elements of $\mathcal F$:

\begin{theorem}\label{main2-intro}
Let $D, G\subset \C^2$ be smooth bounded strongly (linearly) convex domains. Let $\mathcal F$ be a complete collection of complex geodesics.  Then $(D,G,\mathcal F)$ satisfies the slice rigidity property.
\end{theorem}

The smoothness assumption of the boundaries in Theorem~\ref{main-intro} and Theorem~\ref{main2-intro} can be drop up to $C^3$-smoothness, using the boundary representation of  Huang and  Wang  \cite{HW} instead of the one by Chang, Hu and Lee \cite{Cha-Hu-Lee 1988}. We leave this technical point to the interested reader. 

For the proof of Theorem~\ref{main2-intro} we  first prove the corresponding result for $D=G=\B^2$. Then, using a double scaling via squeezing functions and using the  slice rigidity property in $\B^2$, we prove that any holomorphic map $F$ which acts as an isometry on $\mathcal F$ has no singular values close to the boundary of $G$ and that lifts paths, forcing $F$ to be invertible. 

In fact, the previous strategy holds in any dimension and,  if one could prove that the slice rigidity property (for complete collection of complex geodesics, without assumption on the existence of a common fixed point) holds for $\B^n$, $n\geq 3$, then Theorem~\ref{main2-intro} would hold in any dimension. Unfortunately, our proof of the result in $\B^2$ does not seem to generalize to $n\geq 3$. Therefore, the main open question is: does Theorem~\ref{main2-intro} hold for $D=G=\B^n$, $n\geq 3$?

The proof of Theorem~\ref{main-intro} is more complicated than that of Theorem~\ref{main2-intro}. We use a different strategy, and prove that any holomorphic map $F$ which acts as an isometry on $\mathcal F$ is proper. We briefly sketch the idea in case $p\in\partial D$.  First of all, we show that the image of every element in $\mathcal F$ via $F$ contains a point $q\in \partial G$. Then, using stability of complex geodesics, we prove that if $F$ is not proper,  there exists a sequence $\{z_k\}\subset D$ converging complex tangentially to $p$ such that $\{F(z_k)\}$ is relatively compact in $G$. Next, we use  a suitable local scaling in $D$, which fixes $p$,  and allows to   create an open bunch of complex geodesics in $\B^n$ passing through a given point  $e_1\in \partial \B^n$. The same scaling keeps the complex geodesics through $z_k$ and $p$ almost complex tangential in $\B^n$. Then we scale $G$ using squeezing functions along a suitable non-tangential sequence converging to $q$. The maps which result from the composition of these two scalings with $F$ fix the origin and  converge to a holomorphic self-map of the unit ball which acts as an isometry on a collection of complex geodesics satisfying the hypothesis of Theorem~\ref{main-ball}, so an  unitary transformation. But, by construction, this transformation sends an almost complex tangential geodesic into a one very close to the origin, getting a contradiction.
\smallskip

These type of rigidity results are faced naturally when dealing with invariant objects in strongly convex domains, or,  more general, in those domains where every two points can be joined by a complex geodesics. For instance, let $D\subset \C^n$ be a smooth ($C^3$ is in fact enough) bounded strongly pseudoconvex domain. In \cite{Bra-Pat, BPT, BST, HW}, for every $p\in \partial D$ it was defined a pluricomplex Poisson kernel $\Omega_{D,p}$ which shares many properties with the Poisson kernel in $\D$. In  particular (see \cite[Theorem 7.3]{Bra-Pat})  it was proved that if $D, G\subset \C^n$ are smooth bounded strongly convex domains and $F:D\to G$ is a holomorphic map, {\sl which extends continuously at $p\in \partial D$ and $F(p)=q\in \partial G$}, then $F$ is a biholomorphism if and only if there exists $\lambda>0$ such that $F^\ast \Omega_{G,q}=\lambda \Omega_{D,p}$. By \cite{HW}, one can assume $D, G$ to be $C^3$-smooth bounded strongly (linearly) convex domains. The proof of such a result shows that if there exist $p\in \partial D$ and $q\in \partial G$ such that $F^\ast \Omega_{G,q}=\lambda \Omega_{D,p}$ (no continuity extension of $F$ required at $p$) then $F$ maps the complex geodesics of $D$ whose closure contains $p$ isometrically onto complex geodesics of $G$. 
 Thus, applying Theorem\ref{main-intro} we can drop the continuity assumption of $F$ at $p$, and we have:

\begin{corollary}
Let $D, G\subset \C^n$ be $C^3$-smooth bounded strongly (linearly) convex domains. Let $F:D\to G$ be holomorphic. Then $F$ is a biholomorphism if and only if there exist $p\in \partial D$, $q\in \partial G$  and $\lambda>0$ such that $F^\ast \Omega_{G,q}=\lambda\Omega_{D,p}$.
\end{corollary}

Another instance of appearance of the slice rigidity property is in the proof of the strong form of the Schwarz-Pick lemma at the boundary for strongly convex domains \cite[Theorem 2.3, Theorem 7.5]{Bra-Kra-Rot 2020}.  Roughly speaking, such a result says that the only holomorphic maps, between two smooth bounded   strongly convex domains, which preserve the infinitesimal Kobayashi metric up to order one with respect to the Euclidean distance from a given boundary point are biholomorphisms. However, two extra technical conditions  are needed in the proof of such a result in dimension greater than $1$. The proof of \cite[Theorem 7.5]{Bra-Kra-Rot 2020} shows that assuming only hypothesis (2)  in that theorem, one comes up with a map which is an isometry along complex geodesics. The technical hypothesis (1) in \cite[Theorem 7.5]{Bra-Kra-Rot 2020} was needed to prove that the map is then proper, hence a biholomorphism. Using Theorem~\ref{main-intro} we can thus remove such a condition. Hence, one can prove the following result, which improves \cite[Theorem 2.3, Theorem 7.5]{Bra-Kra-Rot 2020}.

Let $D\subset \C^n$ be a bounded domain with smooth boundary. As customary, $\delta_D(z)$ denotes the Euclidean distance of a point $z\in D$ from  $\partial D$. Also, if $z\in D$, we  denote by $\pi(z)\in \partial D$ a point  such that $\delta_D(z)=|z-\pi(z)|$ (such a point is unique if  $z$ is sufficiently close to the boundary). Finally, for $q\in \partial D$ and $v\in \C^n$ we denote by $\Pi_q(v)$ the orthogonal projection (with respect to the standard Hermitian product of $\C^n$) of $v$ onto $T^\C_q\partial D$. With this notation at hand, we have:

\begin{corollary}
Let $D, D'\subset \C^N$ be two bounded strongly convex domains with smooth boundary. Let $F:D\to D'$ be holomorphic.  Let $p\in \partial D$. Then $F$ is a biholomorphism if and only if 
\begin{itemize}
\item[a)]  $|\Pi_{\pi(F(z))}(dF_{z}(w))|$ is uniformly bounded for any $w\in \C^n\setminus T_p\partial D$, $|w|=1$ when $z\to p$ non-tangentially, 
\item[b)] and 
\begin{equation*}
k_{D'}(F(z); dF_{z}(w))=k_D(z;w)+o\left(\delta_D(z)\right), 
\end{equation*}
when $ z\to p$ non-tangentially and uniformly in $w$. 
\end{itemize}
  \end{corollary}

We do not know if condition (a) in the previous corollary is necessary.

\medskip

Part of this work has been done while the first named author was visiting the host institution of the other two authors in Krak\'ow. He wishes to express his thanks for the wonderful atmosphere he experienced there.

\section{The proof of Theorem~\ref{main-ball}}

We start with the following general proposition:

\begin{proposition}\label{prop-analytic}
Let $G\subset \C^n$ be  bounded strongly (linearly) convex domain with real analytic boundary.  Let $\mathcal F$ be a  collection of complex geodesics of $\B^n$ such that there exist a compact set $K\subset \B^n$ and a set $U$ so that $U$ is a neighborhood of some point of $\partial \B^n$  and for every $z\in U\cap \B^n$ there exists $\varphi\in\mathcal F$ such that $\varphi(\D)\cap K\neq\emptyset$ and $z\in \varphi(\D)$. Then $(\B^n,G,\mathcal F)$ satisfies the slice rigidity property.
\end{proposition}
\begin{proof}
 Let $F:\B^n\to G$ be holomorphic with the property that $F\circ \varphi$ is a complex geodesic in $G$ for every $\varphi\in \mathcal F$.

Since  by hypothesis $U\cap \partial \B^n$ is a neighborhood of some point of $\partial \B^n$, there exists a non-empty connected open set $W\subset U\cap \partial \B^n$. 

We claim that, for every $p\in W$, there exists $\varphi_p\in \mathcal F$ such that $\varphi_p(\D)\cap K\neq \emptyset$ and $p\in \varphi_p(\oD)$.
Indeed, let $\{z_k\}\subset \B^n$ be a sequence converging to $p$. By construction, $z_k$ is eventually contained in $U$---and, with no loss of generality, we assume $z_k\in U$ for all $k$ --- hence, there exist $\varphi_k\in\mathcal F$ such that $z_k\in \varphi_k(\D)$ and $\varphi_k(\D)\cap K\neq\emptyset$ for all $k$. The last condition implies that, up to re-parametrization and extracting subsequences, $\varphi_k$ converges uniformly on $\overline{\D}$ to a complex geodesic $\varphi_p$ such that $\varphi_p(\D)\cap K\neq \emptyset$ (this is easy to see for images of complex geodesics of $\B^n$ because they are intersections of complex affine lines with $\B^n$, or, see, {\sl e.g.}, \cite[Corollary 2.3]{BPT}). Since $z_k\in \varphi_k(\D)$, the uniform convergence on $\oD$ implies that $p\in \varphi_p(\oD)$.

Let $p\in W$ and let $\varphi_p$ be as above. Up to re-parameterization, we can assume $\varphi_p(1)=p$. We define
\[
F(p):=\lim_{r\to 1}F(\varphi_p(r))\in\partial G.
\]
In fact, such a limit exists because by hypothesis $F\circ \varphi_p$ is a complex geodesic in $G$, which extends hence real analytically through $\partial \D$. 

Now we show that $F$ is continuous at $p$, for every $p\in W$. To this aim, let $\{z_k\}\subset \overline{\B^n}$ be a sequence converging to $p$. By construction, $z_k$ is eventually contained in $U$---and, with no loss of generality, we assume $z_k\in U$ for all $k$---hence, there exist $\varphi_k\in\mathcal F$ such that $z_k\in \varphi_k(\oD)$ and $\varphi_k(\D)\cap K\neq\emptyset$ for all $k$. Since $\{F\circ \varphi_k\}$ are complex geodesics in $G$, whose images intersect the relatively compact set $F(K)$, up to re-parametrization and extracting subsequences, we can assume that $\{\varphi_k\}$ converges uniformly on $\oD$ to a complex geodesic $\varphi_p$ of $\B^n$ such that $p\in \varphi_p(\oD)$---and we can assume $\varphi_p(1)=p$---and $\{F\circ \varphi_k\}$ converges uniformly on $\oD$  to a complex geodesic $\eta$ of $G$ (see, {\sl e.g.}, \cite[Corollary 2.3]{BPT}). Clearly, $\eta=F\circ \varphi_p$, since $\{F\circ \varphi_k\}$ converges uniformly on compacta to $F\circ \varphi_p$.   Moreover, due to the uniform convergence on $\oD$, if $z_k=\varphi_k(\zeta_k)$ for some $\zeta_k\in \D$, we have
\[
\lim_{k\to \infty}F(z_k)=\lim_{k\to \infty}F(\varphi_k(\zeta_k))=F(\varphi_p(1))=F(p).
\]
Hence, $F$ extends continuously on $W$ and $F(W)\subset\partial G$. In particular, $F$ is (in the distributional sense) a continuous CR-mapping between $W$ and $F(W)\subset \partial G$. 

Therefore, by \cite{PT}, for every $p\in W$ there exists a neighborhood $W_p\subset W$ such that $F:\B^n\cup W_p\to G\cup F(W_p)$ is real analytic. Thus, by \cite[Theorem 2]{Pin}, $F:\B^n \to G$ is a biholomorphism. 
\end{proof}

Now we can prove Theorem~\ref{main-ball}.

\begin{proof}[Proof of Theorem~\ref{main-ball}]
According to Proposition~\ref{prop-analytic}, it is enough to show that there exist a compact set $K\subset \B^n$ and a set $U$ so that $U$ is a  neighborhood of some point of $\partial \B^n$  and for every $z\in U\cap \B^n$ there exists $\varphi\in\mathcal F$ such that $\varphi(\D)\cap K\neq\emptyset$ and $z\in \varphi(\D)$.

Let $V\subset \B^n$ be the open set given by the hypothesis of Theorem~\ref{main-ball}. Let $B$ be a small Euclidean ball such that $K:=\overline{B}\subset V\setminus\{q\}$. 

For each $z\in K$, let $v_z:=z-q$. Hence, $(q+\C v_z)\cap \B^n$ is (the image of) a complex geodesic in $\B^n$ whose closure contains $q$ and $z$. By hypothesis, and since complex geodesics in $\B^n$ are unique up to parametrization, it follows that for every $z\in K$ there exists $\varphi\in \mathcal F$ such that $\varphi(\D)=(q+\C v_z)\cap \B^n$. 

Hence, if we let $U:=\bigcup_{z\in K}(z+\C v_z)$, in order to conclude the proof, we need to show that there exists $p\in \partial \B^n$ such that $U\cap \partial \B^n$ is a neighborhood of $p$. 

To see this, up to composing with automorphisms of $\B^n$ we can assume that $K=\{z\in\B^n: |z|\leq \rho\}$ for some $\rho>0$  and, up to composing further with a unitary transformation, we can assume that $q=-re_1$ for $r\in (\rho, +\infty)$, where  $e_1:=(1,0,\ldots, 0)$. Let $Y:=\{w\in \C^{n-1}: |w|<\rho\}$ and let $\Psi:\C\times Y\to \C^n$ be the holomorphic map given by $\Psi(\zeta,w)=-re_1+\zeta (r,w)$, $\zeta\in \C, w\in Y$. Clearly, $\Psi(\C\times Y)\subset U$. Moreover,  a direct computation shows that $\det(d \Psi_{(\zeta, w)})=r\zeta^{n-1}$. In particular, $\Psi$ is a local biholomorphism at $(1+\frac{1}{r},0)$. Since $\Psi(1+\frac{1}{r},0)=e_1$, it follows that $\Psi(\C\times Y)$---and hence $U$---contains an open set $W$ such that $p\in W\cap \partial \B^n$, and we are done.
\end{proof}

\section{The proof of Theorem~\ref{main2-intro}}

The aim of this section is to prove Theorem~\ref{main2-intro}. We first prove the result in case $D=G=\B^2$. We start with the following lemma:

\begin{lemma}\label{Cart-ball}
Let $F:\B^2\to \B^2$ be holomorphic. Suppose there exist two complex geodesics $\varphi, \eta:\D \to \B^2$ such that $\varphi(\D)\cap \eta(\D)=\{z\}$ for some $z\in \B^2$ and $F$ is an isometry for the Kobayashi metric on $\varphi$ and $\eta$. Then $F$ is an automorphism of $\B^2$.
\end{lemma}
\begin{proof}
Composing with automorphisms, we can assume that $z=0$, $\varphi(\zeta)=(\zeta, 0)$, $F(\varphi(\zeta))=(\zeta,0)$ and $\eta(\zeta)=\zeta \underline{v}$ for some $\underline{v}=(v_1,v_2)\in \C^2$ such that $|\underline{v}|=1$ and $v_2\neq 0$. Hence, 
\[
dF_0=\left(\begin{matrix}
1& a\\ 0& b
\end{matrix}\right),
\]
for some $a, b\in \C$. By Schwarz' Lemma, $|dF_0((\alpha, \beta))|\leq 1$ for all $\alpha,\beta\in \C$ such that $|\alpha|^2+|\beta|^2=1$. The previous inequality implies $|\alpha+a\beta|^2+|b\beta|^2\leq 1$, or, equivalently, 
\[
2\Re (\overline{\alpha}a\beta)\leq (1-|a|^2-|b|^2)|\beta|^2.
\]
The only possibility is thus $a=0$, and $|b|\leq 1$. However, if $|b|<1$, then $|dF_0(\underline{v})|<1$, against the hypothesis that 
\[
1=k_{\B^2}(0; \underline{v})=k_{\B^2}(0; dF_0(\underline{v}))=|dF_0(\underline{v})|.
\]
Therefore, $|b|=1$ and by Cartan's Theorem, $F$ is an automorphism of $\B^2$.
\end{proof}

\begin{proposition}\label{ball}
Theorem~\ref{main2-intro} holds if $D=G=\B^2$.
\end{proposition}
\begin{proof}
By Lemma~\ref{Cart-ball}, we can assume that for every $\varphi, \eta\in \mathcal F$ then either $\varphi(\D)=\eta(\D)$ or $\varphi(\D)\cap\eta(\D)=\emptyset$. 

Up to composing with  automorphisms of $\B^2$, we can assume that $0\in \varphi_0(\D)$, with $\varphi_0(\zeta)=(\zeta,0)$, $\zeta\in \D$ and $F(\zeta,0)=(\zeta,0)$. Arguing as in the proof of Lemma~\ref{Cart-ball}, we see that
\begin{equation}\label{Eq:diff}
dF_0=\left(\begin{matrix}
1& 0\\ 0& b
\end{matrix}\right)
\end{equation}
for some $|b|\leq 1$. Note that by Cartan's Theorem, $|b|=1$ if and only if $F$ is an automorphism, so, arguing by contradiction, we can assume $|b|<1$ (and $F$ is not an automorphism). 

Since $\mathcal F$ is complete, for $z\in \B^2$, there exists $\varphi_z\in \mathcal F$ such that, up to re-parametrization, $\varphi_z(0)=z$. 

Since $\varphi_z(\D)$ cannot intersect $\varphi_0(\D)$, it follows that $\varphi_z'(0)\to (1,0)$ as $z\to 0$. 

Let $z=(0,w)$ and $(a_w,b_w):=\varphi_{(0,w)}'(0)$. Note that $a_w\neq 0$ for all $w\in \D$, for otherwise the complex geodesic through $(0,w)$ would intersect $\varphi_0$ at $(0,0)$.

Let $\lambda(w):=\frac{b_w}{a_w}$, which is well defined for $w\in \D$  and $\lambda(w)\to 0$ as $w\to 0$ by the previous remark. 

By hypothesis, $k_{\mathbb B^2} ((0,w); \varphi_{(0,w)}'(0)) = k_{\B^2}(F((0,w)); dF_{(0,w)}(\varphi_{(0,w)}'(0)))$. Therefore, for all $w\in \D$,
\begin{equation}\label{Eq1}
k_{\mathbb B^2}((0,w); (1,\lambda(w)))=k_{\B^2}(F((0,w)); dF_{(0,w)}(1,\lambda(w))).
\end{equation}
Let $u:\D\times \C\to \R$ be defined as
\[
u(w,\lambda) :=  \left(k_{\mathbb B^2}((0,w); (1,\lambda))\right)^2 -  \left( k_{\B^2}(F((0,w)); dF_{(0,w)}(1,\lambda))\right)^2.
\]
Since $F$ does not increase the Kobayashi metric, $u(w,\lambda)\leq 0$, and, by \eqref{Eq1}, $u(w,\lambda(w))\equiv 0$ for $w\in \D$. Hence, $v(w,\lambda):=\frac{\partial u(w,\nu)}{\partial \overline{\nu}}|_{\nu=\lambda}$ has the property that $v(w,\lambda(w))=0$ for $w\in \D$. 

Now, denoting by $\langle \cdot, \cdot\rangle$ the usual Hermitian product in $\C^2$, we have
\begin{equation*}
\begin{split}
u(w,\lambda)&=\frac{1}{(1-|w|^2)^2}[|\lambda|^2|w|^2+(1-|w|^2)(1+|\lambda|^2)]\\&-\frac{1}{(1-|F(0,w)|^2)^2}[|\langle F(0,w), dF_{(0,w)}(1,\lambda)\rangle|^2+(1-|F(0,w)|^2)|dF_{(0,w)}(1,\lambda)|^2].
\end{split}
\end{equation*}
Hence, $v(w,\lambda)=A(w)\lambda+B(w)$
 for some analytic functions $A, B: \D \to \C$. Here analytic means that the real and imaginary parts of $A, B$ are real analytic functions. Moreover,
\[
\frac{\partial u(w,\nu)}{\partial \nu}|_{\nu=\lambda}=\overline{A(w)} \overline{\lambda}+\overline{B(w)}.
\]
 If $A(w_0)=0$ for some $w_0\in \D$, since $\lambda=\lambda(w_0)$ is a solution of $v(w_0,\lambda)=0$, it turns out that also $B(w_0)=0$, and hence $v(w_0, \lambda)=0$---and  $\frac{\partial u(w_0,\nu)}{\partial \nu}|_{\nu=\lambda}= 0$---for all $\lambda\in \C$. 
 
Let $\lambda\neq 0$ and consider the real function $f(t):=u(w_0, t\lambda(w_0))$. Then for all $t$,
\[
f'(t)=\frac{\partial u(w_0,\nu)}{\partial \nu}|_{\nu=t\lambda}\lambda+\frac{\partial u(w_0,\nu)}{\partial \overline{\nu}}|_{\nu=t\lambda}\overline{\lambda}=0.
\]
Hence, $f(t)=f(0)=0$ for all $t$. But this means that $F$ acts as an isometry on all complex geodesic through $w_0$ with direction $(1,t\lambda)$ for all $t$. By Lemma~\ref{Cart-ball}, $F$ is an automorphism, contradicting our assumption. Therefore, $A(w)\neq 0$ for all $w\in \D$ and then $\lambda(w)=-B(w)/A(w)$ is analytic on $w\in \D$ for all $w\in \D$.
 
Let $\Psi: \D\times \C \to \C^2$ be defined as $\Psi(w,\zeta):=(\zeta, w+\zeta\lambda(w))$. Let $\rho\in (0,1)$ and Let  $U_\rho:= \Psi(\rho \D,\C)$. Since the image of any complex geodesic of $\B^2$ is the intersection of $\B^2$ with an affine complex line, we see that 
\[
U_\rho\cap \B^2= \bigcup_{\varphi\in \mathcal F, \varphi(\D)\cap \{(0,w): |w|<\rho\}\neq \emptyset} \varphi(\D).
\]
We claim that  there exists $\rho\in (0,1)$ and $p\in\partial \B^2$ such that $U_\rho$ is an open neighborhood of $p$. 

Assuming the claim for the moment, we have that the collection of $\varphi\in \mathcal F$ which intersect the compact set  $\{(0,w) : |w|\leq \rho\}$ and the open set $U_\rho$ satisfy the hypotheses of Proposition~\ref{prop-analytic}. Hence,  $F$ is an automorphism.

In order to prove the claim, for $w\in \D$ fixed, let $S(w)\subset \C$ be such that $\Psi(w,\zeta)\in \partial \B^2$ for $\zeta\in S(w)$. We show that there exist  $w\in\D$ and $\zeta\in S(w)$ so that  $\Psi$ is a local diffeomorphism at $(w,\zeta)$. 
 
To this aim, let  $J(w,\zeta)$ be the real Jacobian of $\Psi$ at $(w,\zeta)$. We prove, and it is enough, that there exists $w\in \D$ and $\zeta\in S(w)$ such that $J(w,\zeta)\neq 0$. Assume by contradiction that  $J(w,\zeta)= 0$ for all $w\in \D$ and $\zeta\in S(w)$. A direct computation shows that 
 \begin{equation}\label{Jac-2d}
 0\equiv J(w,\zeta)=\left|1+\zeta \lambda_w\right|^2-\left|\zeta \lambda_{\overline{w}}\right|^2,
 \end{equation}
 where  $\lambda_w:=\frac{\partial \lambda}{\partial w}$ and $\lambda_{\overline{w}}:=\frac{\partial \lambda}{\partial \overline{w}}$. Therefore, we have the following system:
\begin{equation*}
\begin{cases}
 |\zeta|^2(1+|\lambda|^2)+|w|^2+2\Re (\zeta \lambda \overline{w})&=1,\\
 |\zeta|^2(|\lambda_w|^2-|\lambda_{\overline{w}}|^2)+1+2\Re(\zeta\lambda_{w})&=0.
\end{cases}
\end{equation*}
Hence, for $\zeta\in S(w)$,
\[
\frac{1-|w|^2}{1+|\lambda|^2}-\frac{1}{|\lambda_{\overline{w}}|^2-|\lambda_w|^2}=2\Re\left(\frac{\zeta \lambda \overline{w}}{1+|\lambda|^2}+\frac{\zeta\lambda_{w}}{|\lambda_{\overline{w}}|^2-|\lambda_w|^2}   \right).
\]
Since the left-hand side does not depend on $\zeta$, the only possibility is
\begin{equation}\label{L-sys}
\begin{cases}
1-|w|^2&=\frac{1+|\lambda|^2}{|\lambda_{\overline{w}}|^2-|\lambda_w|^2},\\
\lambda_{w}&= \frac{(|\lambda_w|^2-|\lambda_{\overline{w}}|^2)}{1+|\lambda|^2}\lambda \overline{w},
\end{cases}
\end{equation}  
that is,
\begin{equation*}
\lambda_w=-\frac{\overline{w}}{1-|w|^2}\lambda.
\end{equation*}
 Therefore,  
\[
\lambda(w)=\overline{\alpha(w)}(1-|w|^2),
\]
for some holomorphic function $\alpha$ defined on $\D$ such that $\alpha(0)=0$. Hence,
\[
|\lambda_{\overline{w}}|^2-|\lambda_w|^2=(1-|w|^2)[|\alpha'(w)|^2(1-|w|^2)-2\Re(w\overline{\alpha(w)}\alpha'(w))].
\]
Substituting in the first equation of \eqref{L-sys} we obtain
\[
(1-|w|^2)^2[|\alpha'(w)|^2(1-|w|^2)-|\alpha(w)|^2-2\Re(w\overline{\alpha(w)}\alpha'(w))]=1,
\]
that is, for all $w\in \D$,
\begin{equation}\label{ineq1}
|\alpha'(w)|^2=|\alpha'(w)w+\alpha(w)|^2+\frac{1}{(1-|w|^2)^2}\geq \frac{1}{(1-|w|^2)^2}.
\end{equation}
This implies that $\alpha':\D \to \C$ is a holomorphic and proper function. But this is not possible, and we reach a contradiction.  
\end{proof}

Now we pass to the general case, proving first a lemma. As a matter of notation, for $\epsilon>0$ and $D\subset \C^n$ a domain, we let
\begin{equation}\label{tub-n}
N_\epsilon(D):=\{z\in D: \hbox{dist}(z,\partial D)<\epsilon\}.
\end{equation}

\begin{lemma}\label{non-dic-close-bd}
Let $D, G, \mathcal F, F$ satisfy the hypotheses of Theorem~\ref{main2-intro}. Then there exists $\epsilon>0$ such that  every $w\in N_\epsilon(G)\cap F(D)$ is a regular value of $F$ ({\sl i.e.}, $\det dF_z\neq 0$ for every $z\in D$ such that $F(z)=w$).
\end{lemma}

\begin{proof}
We argue by contradiction. Assume the result is not true, we can find a sequence $\{w_n\}\subset G$ converging to the boundary of $G$ and a sequence $\{z_n\}\subset D$ such that $F(z_n)=w_n$ and $\det dF_{z_n}=0$. Since $F(D)\subset G$, the sequence $\{z_n\}$ cannot be relatively compact in $D$, and we can assume it converges to the boundary of $D$.

Let $s_D(z)$ denote the {\sl squeezing function of $D$ with respect to $z\in D$}, that is,
\begin{equation}
s_D(z):=\sup\{r>0:r\mathbb B^n\subset f(D),f:D\to \B^n \hbox{\ univalent}, f(z)=0\}.
\end{equation}
 Recall that, by  \cite{Die-For-Wold 2014}, $\lim_{z\to\partial D}s_D(z)=1$ for any strongly pseudoconvex domain $D\subset\mathbb C^n$ with $C^2$-smooth boundary  (see Theorem 4.1 in \cite{Kim-Zha 2016}).

Since $s_D(z_n) \to 1$, there exist biholomorphisms $\Phi_n:D\to \Phi_n(D)=:D_n$   such that $r_n \mathbb B^2\subset D_n\subset \B^2$, $\Phi_n(z_n) =0$, where $r_n>0$ and $r_n\to 1$ as $n\to \infty$. 

Similarly, there exist  biholomorphisms $\Psi_n:G\to \Psi_n(G)=:G_n$   such that $r_n \mathbb B^2\subset G_n\subset \B^2$, $\Psi_n(w_n) =0$, where $r_n>0$ and $r_n\to 1$ as $n\to \infty$.

Let $F_n:=\Psi_n\circ F\circ \Phi_n^{-1}: D_n \to G_n$. Note that $F_n(0)=0$ and $\det d(F_n)_0=0$. Therefore, a Montel's argument implies that, up to subsequences, $\{F_n\}$ converges on compacta to a holomorphic map $H:\B^2\to \B^2$ such that $H(0)=0$ and $\det dH_0=0$. 

 Let $\mathcal F_n:=\Phi_n\circ \mathcal F$. Then $\mathcal F_n$ is a complete collection of complex geodesics of $D_n$ for all $n$.  Take $z\in \B^2$. For $n$ sufficiently large, $z\in D_n$. Hence, there exists a complex geodesic $\varphi_n\in \mathcal F_n$ such that $z\in \varphi_n(\D)$. Up to re-parameterization, we can assume $\varphi_n(0)=z$ for all $n$ (sufficiently large). Up to extracting subsequences, we can assume that the sequence $\{\varphi_n\}$ converges to a holomorphic map $\varphi:\D \to \B^2$. Since $\{D_n\}$ exhausts $\B^2$, it follows that $\varphi$ is a complex geodesic of $\B^2$ and the collection $\mathcal H$ of complex geodesics $\varphi$ of $\B^2$ constructed in this way is complete. Moreover, since  $F_n\circ \varphi_n$ is a complex geodesic of $G_n$ for all $n$, by the same argument, we obtain that
 $H\circ \varphi$ is a complex geodesic of $\B^2$. Therefore, $H:\B^2\to \B^2$ is a holomorphic map which acts as an isometry on the complete collection $\mathcal H$ of complex geodesics of $\B^2$, and by Proposition~\ref{ball}, $H$ is an automorphism and in particular $\det dH_0\neq 0$.  A contradiction, and we are done. 
\end{proof}

\begin{proof}[Proof of Theorem \ref{main2-intro}]
Let $F:D\to G$ be holomorphic and acting as an isometry on a complete collection of complex geodesics $\mathcal F$ of $D$. 

 Since $G$ is homeomorphic to the ball, we can find an exhaustion of compact sets $\{K_n\}$ of $G$ such that $G\setminus K_n$ is simply connected. 
 
We first show that there exists $n\in \N$ such that $G\setminus K_n\subset F(D)$. 

Assume by contradiction this is not true. Then we can find a sequence $\{w_n\}\subset G\setminus K_n$ such that $w_n\not\in F(D)$ and $w_n$ converges to some boundary point of $G$. 

Let $\varphi\in \mathcal F$. Since $F\circ \varphi$ is a complex geodesic in $G$, it follows that $F(\varphi(\D))\cap (G\setminus K_n)\neq \emptyset$, for every $n$. Hence, for every $n$ we can find $z_n^0$ such that  $F(z_n^0)\in G\setminus K_n$. Let $\gamma_n:[0,1]\to G\setminus K_n$ be a continuous curve such that $\gamma_n(0)=F(z_n^0)$ and $\gamma_n(1)=w_n$. 

If $n$ is sufficiently big,  by Lemma~\ref{non-dic-close-bd}, $F$ is a local biholomorphism at $z_n^0$, so there exists a continuous curve $\eta_n:[0, t_0^n)\to D$, $0<t_0^n\leq 1$ such that $F(\eta_n(t))=\gamma_n(t)$ for all $t\in [0, t_0^n)$. There are two cases: either the cluster set of $\eta_n$ at $t_0^n$ is contained in $\partial D$ or it contains a point $z\in D$. In the latter case, it follows that $F(z)=\gamma_n(t_0^n)$, but then, again by by Lemma~\ref{non-dic-close-bd}, $dF_z$ is invertible and the path $\gamma_n$ can be lifted through $\gamma_n(t_0^n)$. Therefore, if the first case does not occur, $\eta_n$ is well defined on $[0,1]$ and $F(\eta_n(1))=w_n$. 

Hence, we can assume that  the cluster set of $\eta_n(t)$ for $t\to {t^n_0}^-$ is contained in $\partial D$ and $F\circ \eta_n(t)=\gamma_n(t)$ for $t\in [0,t^n_0)$. 

Since the squeezing function of $G$ is uniform (see Theorem 4.1 in \cite{Kim-Zha 2016}), for every $r\in (0,1)$ there exists $\delta>0$ such that, for each $w\in G$ so that $\hbox{dist}(w, \partial G)<\delta$ there exists a univalent map $\Psi_w:G\to \B^2$ with the property that $\Psi_w(w)=0$ and $r\B^2\subset \Psi_w(G)$. 

Thus, for every $n\in \N$ there exists $m_n\in \N$ such that for every $w\in G\setminus K_{m_n}$, $(1-\frac{1}{n})\B^2\subset \Psi_w(G)$. In order to avoid using double indices, we can relabelling $K_n$ so that $m_n=n$.

Fix $\delta\in (0,\frac{1}{4})$. We claim that for $n\in \N$, there exists $s_n\in (0, t^n_0)$  such that 
\begin{equation}\label{eq-all-rest}
\Psi_{\gamma_n(s)}(\gamma_n(r))\in \delta \B^2
\end{equation}
 for all $s_n\leq s\leq r\leq  t^n_0$. 
 
 In order to prove \eqref{eq-all-rest}, fix $n\in \N$. Let $R>0$ and let $B^K(\gamma_n(t^n_0), R)$ be the Kobayashi ball in $G$ of center $\gamma_n(t^n_0)$ and radius $R$. Let $s_n\in (0, t^n_0)$ be such that $\gamma_n(s)\in B^K(\gamma_n(t^n_0), R)$ for all $s\in [s_n, t^n_0]$. By the triangle inequality,  for all $u,v \in [s_n, t^n_0]$, $\gamma_n(u)\in B^K(\gamma_n(v), 2R)$. Let $s\in [s_n, t^n_0]$. Since $\Psi_{\gamma_n(s)}$ is a biholomorphism on its image, it follows that $\Psi_{\gamma_n(s)}(B^K(\gamma_n(s), 2R))$ is a Kobayashi ball in $\Psi_{\gamma_n(s)}(G)$, centered at $0$ and with (hyperbolic) radius $2R$. Taking into account that $\Psi_{\gamma_n(s)}(G)\subset \B^2$ and the monotonicity of the Kobayashi distance, it follows that $\Psi_{\gamma_n(s)}(B^K(\gamma_n(s), 2R))$ is contained in the Kobayashi ball of $\B^2$ of center $0$ and radius $2K$. Since Kobayashi balls in $\B^2$ centered at $0$ are just balls, by taking $K$ sufficiently small we are done.

Using the same uniform squeezing argument  in $D$, and since $\eta_n(t)\to \partial D$ as $t\to t^n_0$, we can find $t_n\in [s_n, t^n_0)$ such that there exists a univalent map $\Phi_n:D \to \B^2$ so that $\Phi_n(\eta_n(t_n))=0$ and $(1-\frac{1}{n})\B^2\subset D_n:=\Phi_n(\B^2)$. 

Let $\Psi_n:=\Psi_{\gamma_n(t_n)}$, $G_n:=\Psi_n(G)$. Let $F_n:=\Psi_n \circ F \circ \Phi_n^{-1}: D_n\to G_n$. Hence, $\{F_n\}$ is a sequence of holomorphic maps such that $F_n(0)=0$ and $(1-1/n)\B^2\subset D_n\subset \B^2$, $(1-1/n)\B^2 \subset G_n\subset \B^2$. 

By Montel's theorem, we can assume that $\{F_n\}$ converges uniformly on compacta to a map $H:\B^2\to \B^2$. 

Arguing as in  the proof of Lemma~\ref{non-dic-close-bd}, we see that $H$ acts as an isometry on a complete collection of complex geodesics of $\B^2$, hence, by Proposition~\ref{ball}, it is an automorphism of $\B^2$ such that $H(0)=0$. Therefore, since $|H(z)|=|z|$ for all $z\in \B^2$, we deduce that for $n$ sufficiently large, $|F_n(z)|>1/2$ for all $z\in \B^2$ such that $|z|=2/3$. However, since $\eta_n(t)\to \partial D$ as $t\to t^n_0$, for every $n$ we can find $r_n\in (t_n, t^n_0)$ such that $|\Phi_n(\eta_n(r_n))|=2/3$. But, by \eqref{eq-all-rest}, we have $F_n(\Phi_n(\eta_n(r_n)))\in \delta\B^2$, obtaining a contradiction. 

Therefore, for $n$ sufficiently large, $G\setminus K_n\subset F(D)$ and  for every $z\in D$ such that $F(z)\in G\setminus K_n$, $F$ is a local biholomorphism. The previous argument shows also that every continuous path $\gamma:[0,1]\to G\setminus K_n$ can be lifted by $F$ to a continuous path in $D$. Since $G\setminus K_n$ is simply connected, this allows to define a holomorphic map $F^{-1}: G\setminus K_n \to D$, which has the property that $F\circ F^{-1}={\sf id}$. By Hartogs' extension theorem, the map $F^{-1}$ extends to all $G$, and we are done.
\end{proof}

\section{The proof of Theorem~\ref{main-intro}}

We need a couple of preliminary lemmas.

\begin{lemma}\label{Lem-bd-geo-dom}
Let $D\subset \C^n$ be a smooth bounded strongly (linearly) convex domain. Let $p\in \partial D$ and let $\varphi$ be a complex geodesic in $D$ such that $\varphi(1)=p$. Let $\epsilon>0$. Then for any complex geodesic $\psi$ of $D$ such that $\psi(1)=p$ there exists $t_0=t_0(\varphi,\psi,\epsilon) \in (0,1)$ such that for every $t\in (t_0, 1)$     there exists $s_t\in (0,1)$ so that
\[
K_{D}(\varphi(t), \psi(s_t))\leq \epsilon.
\]
\end{lemma}
\begin{proof} We can assume that $\B^n\subset D$ and that $p=e_1=(1,0,\ldots, 0)$. Let $\pi:\C^n\to \D\times\{0\}$ be the orthogonal projection. 

For $M>1$, let 
\[
S(M):=\{\zeta\in \D: |1-\zeta|\leq M(1-|\zeta|)\},
\]
a {\sl Stolz angle} of vertex $1$ and amplitude $M$. 

In the sequel we will use the Chang-Hu-Lee parametrization \cite{Cha-Hu-Lee 1988}: let 
\[
L_{e_1}:=\{v\in \C^n: |v|=1, v_1>0\}. 
\]
Chang, Hu and Lee proved that every complex geodesic $\eta$ of $D$ whose closure contains $p$ can be parameterized in such a way that $\eta(1)=e_1$ and $\eta'(1)=v_1 v$ with $v\in L_{e_1}$. 

In the rest of the proof we denote by $\mathcal C_{e_1}$ the set of all complex geodesics in $D$ parametrized as above.

Note that  $\eta(t)$ converges non-tangentially to $e_1$ (by the previous parametrization or by Hopf's Lemma) for all $\eta\in \mathcal C_{e_1}$. In particular, $\eta(t)$ is eventually contained in $\B^n$.

Also, note that $\pi(\eta'(t))\to v_1^2>0$ as $t\to 1$ for all $\eta\in \mathcal C_{e_1}$, for some $v\in L_{e_1}$. This implies immediately that for every $M>1$ and for every $\eta\in \mathcal C_{e_1}$ there exists $t_0\in [0,1)$ such that  $\pi(\eta(t))$ belongs to $S(M)\times\{0\}$ for $t\in [t_0, 1)$.

Now, a curve $\gamma(t)\subset \D$ converging to $1$ as $t\to 1$ is eventually contained in $S(M)$ if and only if there exists $R=R(M)>0$ such that for every fixed $r\in [0,1)$ there exists $t_r\in (0,1)$ so that $K_\D(\gamma(t), [r,1))\leq R$ for $t\in [t_r,1)$ (see \cite[Section 6.2]{BCD}), and $R\to 0$ as $M\to 1$. Therefore,  it follows that for every $\delta>0, r\in (0,1)$ and $\eta\in \mathcal C_{e_1}$ there exists $t_1\in [0,1)$ so that for all $t\in [t_1,1)$,
\begin{equation}\label{Eq:B1-bdg}
K_{\B^n}(\pi(\eta(t)), [r,1))=K_{\D\times\{0\}}(\pi(\eta(t)), [r,1))\leq \delta.
\end{equation}
Note that, since the curve $\pi(\eta(t))$ is continuous and converges to $e_1$, then for every $t_2\in (0,1)$ there exists some $r_1\in [r,1)$ so that for every $s\in [r_1,1)e_1$ there exists $t_s\geq t_2$ such that 
\begin{equation}\label{Eq:B1-bdg-rev}
K_{\B^n}(se_1, \pi(\eta(t_s)))=K_{\D\times\{0\}}(se_1, \pi(\eta(t_s)))\leq \delta.
\end{equation}

Finally, since $\eta(t)$ is (eventually) contained in $\B^n$ and converges to $e_1$ non-tangentially,  by \cite[Lemma 2.3]{BF},
\begin{equation}\label{Eq-proj-0}
\lim_{t\to 1}K_{\B^n}(\eta(t), \pi(\eta(t)))=0.
\end{equation} 

Now, let $\varphi, \psi\in \mathcal C_{e_1}$ and $\epsilon>0$. For $t, s$ sufficiently close to $1$ so that $\psi(s), \varphi(t)\in \B^n$, we have
\begin{equation*}
\begin{split}
K_D(\varphi(t), \psi(s))&\leq K_{\B^n}(\varphi(t), \psi(s))\leq \\& K_{\B^n}(\varphi(t), \pi(\varphi(s)))+K_{\B^n}(\psi(t), \pi(\psi(s)))+K_{\B^n}(\pi(\varphi(t)), \pi(\psi(s))).
\end{split}
\end{equation*}
The first two terms in the right side of the inequality are as small as we want for $s, t$ close to $1$, say they are each less than $\epsilon/4$ for $t\geq t'$ and $s\geq s'$ for some $t',s'\in (0,1)$. As for the term $K_{\B^n}(\pi(\varphi(t)), \pi(\eta(s)))$,  we first find $r_1$ so that \eqref{Eq:B1-bdg-rev} holds for $\delta=\epsilon/4$, $\eta=\psi$  and $t_2=s'$, and then we choose $t_0\geq t'$  so that \eqref{Eq:B1-bdg} holds for $\delta=\epsilon/4$, $\eta=\varphi$ and $r=r_1$. 

Hence, given $t\geq t_0$, we can find $s\in [r_1, 1)$ so that $K_{\B^n}(\pi(\varphi(t)), se_1)\leq \epsilon/4$. Then we can find $s_t\geq s'$ such that $K_{\B^n}(\pi(\eta(s_t)), se_1)\leq \epsilon/4$. By the triangle inequality, we have the statement.
\end{proof}

\begin{remark}\label{cono}
The previous proof (and \cite[Lemma 2.3]{BF}) shows that the if $\{\psi_k\}$ is a collection of complex geodesics in $D$ such that $\psi_k(1)=p$ for all $p$ and $\psi_k'(1)$ belongs to a fixed cone of vertex $p$ and amplitude less than $\pi/2$ then number $t_0$ can be chosen independently of $\psi_k$. 
\end{remark}

\begin{remark}\label{cono2}
The previous proof (and \cite[Lemma 2.3]{BF}) also shows that the if $\gamma:[0,1)\to D$ is a continuous curve such that $\gamma$ converges non-tangentially to $p$ as $t\to 1$ then there exists $C_2>0$ such that for every $s\in [0,1)$ there exists $t_s\in [0,1)$ such that $K_D(\gamma(s), \varphi(t_s))<C_2$.
\end{remark}

\begin{lemma}\label{Lem:control-dist}
Let $D\subset \C^n$ be a smooth bounded strongly (linearly) convex domain. Let $p\in \partial D$ and let $\varphi_0$ be a complex geodesic in $D$ such that $\varphi_0(1)=p$. Let $\gamma:[0,1)\to D$ be a continuous curve such that $\lim_{t\to 1}\gamma(t)=p$ and assume that there exists $C>0$ such that for all $t\in [0,1)$,
\begin{equation}\label{Eq:control-nontg}
K_D(\gamma(t), \varphi_0([0,1))\leq C.
\end{equation}
Finally, let $\{\eta_k\}$ be a sequence of complex geodesics in $D$ such that $\eta_k(1)=p$ for all $k$ and $\{\eta_k\}$ converges uniformly on $\oD$ to a complex geodesic $\eta$ of $D$. Then, there exists $M>0$ such that for every $t\in [0,1)$ and every $k\in \N$ there exists $\xi_{t,k}\in (0,1)$ such that
\[
K_D(\eta_k(\xi_{t,k}), \gamma(t))\leq M.
\]
\end{lemma}
\begin{proof}
By Lemma~\ref{Lem-bd-geo-dom} and Remark~\ref{cono}, there exists $C_1>0$ such that for all $t\in [0,1)$ and for all $k\in \N$ there exists $\xi_{t,k}\in (0,1)$ such that $K_D(\eta_k(\xi_{t,k}), \varphi(t))\leq C_1$. Hence, the statement follows at once  by \eqref{Eq:control-nontg} and the triangle inequality. 
\end{proof}

\begin{proof}[Proof of Theorem~\ref{main-intro}]
Let $F:D\to G$ be holomorphic and such that it acts as an isometry on $\mathcal F$. 
\medskip

{\sl Step 1.} There exists $p'\in\overline{ G}$ such that $p'\in \overline{F(\varphi(\D))}$ for all $\varphi\in \mathcal F$. Moreover, $p'\in \partial G$ if and only if $p\in\partial D$.

\smallskip

This is obvious if $p\in D$. If $p\in \partial D$, let $\varphi,\eta\in \mathcal F$ and assume, up to re-parametrization, that $\varphi(1)=\eta(1)=p$. 

Now, since by hypothesis $F\circ \eta$ is a complex geodesic in $G$, there exists $p'\in \partial G$ such that $\lim_{\zeta\to 1}F(\eta(\zeta))=F(\eta(1))=p'$ and similarly there exists $x\in \partial G$ such that $\lim_{\zeta\to 1}F(\varphi(\zeta))=F(\varphi(1))=x$. Fix $\epsilon>0$. Hence,  by Lemma~\ref{Lem-bd-geo-dom},
\[
\lim_{t\to 1}K_G(F(\varphi(t)), F(\eta(s_t)))\leq \lim_{t\to 1}K_D(\varphi(t), \eta(s_t))\leq \epsilon.
\]
By the standard estimates on the Kobayashi distance in strongly convex domains (see, {\sl e.g.} \cite{A2}), it follows at once that $s_t\to 1$ as $t\to 1$ and that $x=p'$. By the arbitrariness of $\eta, \varphi$, Step 1 is  proved.

\medskip

{\sl Step 2.} For every $q\in \partial D\setminus\{p\}$ there exists $q'\in \partial G\setminus\{p'\}$ such that $\lim_{z\to q}F(z)=q'$. 

\smallskip

Take a sequence $\{z_k\}\subset D$ converging to $q$ and consider complex geodesics $\varphi_k$ in $\mathcal F$ containing $z_k$ and $p$. Up to re-parametrization and extracting subsequences, we can assume that $\{\varphi_k\}$ converges uniformly on $\oD$ to a complex geodesic $\varphi_q$ of $D$ such that $p, q\in \varphi_q(\oD)$---and we can assume $\varphi_q(1)=q$ (see, {\sl e.g.}, \cite[Corollary 2.3]{BPT}). Note that $\varphi_q$ does not depend on the  chosen sequence $\{z_k\}$ because, up to parametrization, there exists a unique complex geodesic in $D$ whose closure contains $q$ and $p$. 

By hypothesis, $F\circ \varphi_k$ are complex geodesics in $G$, whose closures contain $p'$ for all $k$, by Step 1. 

We claim that $\{F\circ \varphi_k\}$ converges uniformly on $\oD$ to a complex geodesic in $G$---and, if this is the case, such a complex geodesic is clearly $F\circ \varphi_q$. 

In order to prove the claim, we first note that  the Euclidean diameter of $\{F(\varphi_k(\D))\}$ is uniformly bounded from below from zero. This is clear if $p\in D$. On the other hand, if $p\in\partial D$, and if the Euclidean diameter of some subsequence $\{F(\varphi_{k_m}(\D))\}$ tends to zero, then, taking into account that $p'\in F(\varphi_k(\oD))$ for all $k$, it follows from  \cite[Proposition 2.3]{BFW}  that $F(\varphi_{k_m}(\zeta))\to p'$ as $m\to \infty$, uniformly on $\oD$. But,
\[
K_G(F(\varphi_{k_m}(0)), F(\varphi_q(0))\leq K_D(\varphi_{k_m}(0), \varphi_q(0))\to 0, \quad m\to \infty,
\]
 which immediately  provides a contradiction. 
 
 Therefore, the Euclidean diameter of $\{F(\varphi_k(\D))\}$ is uniformly bounded from below from zero, and  by \cite[Proposition 1]{H} (see also \cite[Section 2]{BPT}) it follows that $\{F\circ \varphi_k\}$ converges uniformly to a complex geodesic of $G$, which is, in fact, $F\circ \varphi_q$.  

Now, let $q':=\lim_{\zeta \to 1}F(\varphi_k(\zeta))=F(\varphi_q(1))$. Note that $q'\neq p'$ because $F\circ \varphi_q$ is injective on $\oD$.

Hence,  there exists a sequence $\{\zeta_k\}\subset \D$ such that $\varphi_k(\zeta_k)=z_k$ and that $\zeta_k\to 1$ as $k\to \infty$. So  $F(z_k)=F(\varphi_k(\zeta_k))\to F(\varphi_q(1))=q'$ as $k\to \infty$ and this proves that $\lim_{z\to p}F(z)=q'$. 

\medskip

{\sl Step 3.} If $p\in \partial D$ and $\{z_k\}\subset D$ is a sequence converging to $p$ non complex tangentially ({\sl i.e.}, $\{\frac{p-z_k}{|p-z_k|}\}$ converges to a vector $v$ such that $\langle v, \nu_p\rangle\neq 0$, where $\nu_p$ is the outer normal unit vector of $\partial D$ at $p$), then $\lim_{k\to \infty}F(z_k)=p'$.

\smallskip

Seeking for a contradiction, we can assume that $\lim_{k\to \infty}F(z_k)=w_0\in G$. For every $k\in \N$, let $\varphi_k\in \mathcal F$ be such that $z_k\in \varphi_k(\D)$. By \cite[Lemma 4.3]{BST} the Euclidean diameter of $\{\varphi_k(\D)\}$ is bounded from below from zero. Hence, by \cite[Proposition 1]{H} (see also \cite[Section 2]{BPT}) we can assume that, up to reparametrization and extracting subsequences, the sequence $\{\varphi_k\}$ converges uniformly on $\oD$ to a complex geodesic $\varphi$ of $D$. In particular, there exists $\zeta_k\in \D$ such that $\varphi_k(\zeta_k)=z_k$. Since $\{\varphi_k\}$ converges uniformly on $\oD$ to $\varphi$, it is easy to see that $\{\zeta_k\}$ converges to a point $\xi\in \partial \D$ and that $\varphi(\xi)=p$. 

By the same token, since  $\{F\circ \varphi_k\}$ are complex geodesics in $G$ by hypothesis and they all intersect a compact subset of $G$ containing $w_0$, we see that $\{F\circ \varphi_k\}$ converges uniformly on $\oD$ to $F\circ \varphi$, which is a complex geodesic in $G$. Then, for all $k\in \N$
\[
K_\D(0,\zeta_k)=K_G(F(\varphi_k(0)), F(\varphi_k(\zeta_k)))=K_G(F(\varphi_k(0)), F(z_k)).
\]
Taking the limit for $k\to \infty$, we see that the left hand side term converges to $\infty$ while the right hand term converges to $K_G(F(\varphi(0)), w_0)<\infty$, a contradiction and Step 3 is proved.

\medskip

{\sl Step 4.} $F$ is proper.
\smallskip

If $p\in D$, then $F$ is proper by Step 2. Therefore, from now on, we assume $p\in \partial D$.

Arguing by contradiction, we assume that there exists a sequence $\{z_k\}\subset D$ converging to $\partial D$ such that $\lim_{k\to \infty}F(z_k)=w_0\in G$. By Step 2, it follows that necessarily $\{z_k\}$ converges to $p$. Moreover, we can assume that $\lim_{k\to \infty}\frac{p-z_k}{|p-z_k|}=v$ exists. By Step 3, $v\in T_p^\C\partial D$, that is, $\langle v, \nu_p\rangle =0$.

Now, we make use of a suitable scaling in $D$ and in $G$. 

To prepare the scaling in $D$,  we can find a local biholomorphism $\Phi$ defined in an open neighborhood $V$ of $p$ such that $\Phi(V\cap D)\subset \B^n$, $e_1=(1,0,\ldots, 0)\in \partial \Phi(V\cap D)$  and the defining function of $\Phi(V\cap D)$ at $e_1$ is given by
\[
r(z)=-1+|z|^2+o(|z-e_1|^2).
\]
We can assume that $d\Phi_p(\nu_p)=\lambda e_1$ for some $\lambda>0$.

Note that by \cite[Corollary 1]{H}, if $\varphi\in \mathcal F$ (and we can assume $\varphi(1)=p$), then there exists $\epsilon>0$ such that if $|\varphi'(1)_N|<\epsilon |\varphi'(1)_T|$ (where, for $v\in \C^n$, $v_N$ denotes the orthogonal projection of $v$ on $\C\nu_p$ and $v_T$ denotes the tangential component of $v$ at $T_p^\C\partial D$ in the orthogonal splitting $v=v_N+v_T$) then $\varphi(\oD)\subset V$. Therefore, there is an open set $\tilde{\mathcal F}^T$ of complex geodesics of $\mathcal F$ which are contained in $D\cap V$. Note that, since each $\varphi\in \tilde{\mathcal F}^T$ has an associated holomorphic retraction $\rho:D\to \varphi(\D)$, it follows that $\rho|_{V\cap D}:V\cap D\to \varphi(\D)$ is a holomorphic retraction of $V\cap D$ onto $\varphi(\D)$, and it follows easily that $\varphi$ is then a complex geodesic of $V\cap D$. Thus, $\Phi \circ \varphi$ is a complex geodesic in $V\cap D$. We let $\mathcal F^T:=\Phi(\tilde{\mathcal F}^T)$.

Also, note that the sequence $\{z_k\}$ is eventually contained in $V$ (and we assume it is contained in $V$ for every $k$) and $\{z_k':=\Phi(z_k)\}$ converges complex tangentially to $e_1$, that is, $\langle \frac{e_1-z_k'}{|e_1-z_k'|}, e_1\rangle\to 0$ as $k\to \infty$. 

Moreover, if $\tilde\varphi_k\in \mathcal F$ are such that $\tilde\varphi_k(1)=p$ and $z_k\in \tilde\varphi_k(\D)$ it follows that $\tilde\varphi_k\in \tilde{\mathcal F^T}$ eventually (and we can assume that they are contained in that set for all $k$) and then, for every $k$,  $\varphi_k:=\Phi\circ \tilde\varphi_k\in \mathcal F^T$ is a complex geodesic in $\Phi(V\cap D)$ which extends smoothly up to the boundary, $\varphi_k(1)=e_1$ and 
\[
\frac{\varphi_k'(1)}{|\varphi_k'(1)|}=\alpha_k e_1+v_k,
\]
for some $\alpha_k>0$ and $|v_k|=1$ such that $\langle v_k, e_1\rangle= 0$ and $\alpha_k\to 0$ as $k\to 0$. We can assume that $v_k\to v_0$ as $k\to \infty$ and $\alpha_k$ is strictly decreasing.

We consider the following (hyperbolic) automorphisms of $\B^n$ (here $(z_1,z')\in \C\times \C^{n-1}$) for $t\in (-1,1)$:
\[
A_t(z_1,z')=\left(\frac{z_1-t}{1-tz_1}, \sqrt{1-t^2}\frac{z'}{1-tz_1}\right)
\]
Note that $-e_1, e_1$ are fixed by $A_t$ for all $t$ and $A_t(t,0)=(0,0)$. We have
\begin{equation}\label{diff}
d(A_t)_{e_1}=\left(\begin{matrix}\frac{1+t}{1-t} & 0\\
0 & \sqrt{\frac{1+t}{1-t}} \end{matrix}\right).
\end{equation}
One can easily see  that $A_t(\Phi(V\cap D))$ converges in the Hausdorff topology to $\B^n$ as $t\to 1$ (or, see, \cite{Fridman}). Moreover, consider the defining functions of $A_t(\Phi(V\cap D))$  given by  
\[
r_t(z):= \frac{|1-tz_1|^2}{1-t^2}r(A^{-1}_t(z)).
\]
It is shown in \cite[pagg. 467-468]{Lem 1981}  that $r_t$ converges in the $C^6$-topology to 
$-1+|z|^2$ on $\{z\in \C^n: \Re z_1\geq -1/2, |z|\leq 2\}$ as $t\to 1$. 
\smallskip

{\sl Claim A}. For every complex geodesic $\eta$ of $\B^n$ such that $\eta(\D)\subset \{\Re z_1>0\}$ and $\eta(1)=e_1$ there exist  a family $\{\eta_t\}\subset \mathcal F^T$ (suitably parametrized)  (defined for $t$ close to $1$) such that $A_t\circ \eta_t$ converges uniformly in the $C^1$-topology on $\oD$ to $\eta$ as $t\to 1$.

\smallskip

For of all, it follows  from \cite[Theorem 2]{Cha-Hu-Lee 1988} that for every $a>0$, $w\in T_p^\C\partial D$ one can find a complex geodesic $\varphi\in\mathcal F$ such that $\varphi'(1)=a\nu_p+w'$. Note that $d\Phi_p(a\nu_p+w')=(\lambda a, Bw')$ for some invertible $(n-1)\times (n-1)$ matrix $B$. 

Now, let $\eta$ be a complex geodesic of $\B^n$ such that $\eta(\D)\subset \{\Re z_1>0\}$ and $\eta(1)=e_1$ and $\eta'(1)= (\alpha, v')$, for some $\alpha>0$, $v'\in \C^{n-1}\setminus\{0\}$. Let $a=\alpha/\lambda$ and $w'=B^{-1}v'$. Let $u_t:=\frac{1-t}{1+t} a\nu_p+ \sqrt{\frac{1-t}{1+t}} w'$. Note that, $\frac{u_t}{|u_t|}$ converges to a vector in $T_p^\C\partial D$. Hence, for $t$ close to $1$ we can find $\tilde\eta_t\in \tilde{\mathcal F}^T$ such that $\tilde\eta_t'(1)=u_t$. Let $\eta_t:=\Phi\circ \tilde\eta_t\in \mathcal F^T$. By \eqref{diff}, $(A_t\circ \eta_t)'(1)=(\alpha, v')$ for every $t$ close to $1$.

Using \cite[Proposition 5]{Cha-Hu-Lee 1988} and reparametrizing if necessary, we see that for every $t$ close to $1$ we can find stationary discs $\eta^\ast_t$ attached to $(A_t(\Phi(\partial D\cap V)))\cap\{\Re z_1>-1/4\}$ such that $\eta_t^\ast(1)=e_1$, $(\eta_t^\ast)'(1)= (a, v')$ and $\{\eta_t^\ast\}$ converges in the $C^1$-topology of $\oD$ to $\eta$. 

Therefore, $h_t:=\Phi^{-1}\circ A_t^{-1}\circ \eta^\ast_t$ is a stationary disc in $D$ such that $h_t(1)=p$ and $h_t'(1)=u_t$. Since, by \cite{Lem 1981}, stationary discs are complex geodesics in $D$, by the uniqueness of complex geodesics (\cite[Theorem 2]{Cha-Hu-Lee 1988}) it follows that $h_t=\tilde\eta_t\circ \theta_t$ for some (parabolic) automorphism $\theta_t$ of $\D$ fixing $1$, and the claim follows.
\smallskip

{\sl Claim B}. For every $R\in (0,1)$  there exists an increasing sequence $\{t_k\}$ converging to $1$ such that, up to reparametrization, $\{A_{t_k}\circ \varphi_k\}$ converges uniformly in the $C^1$-topology on $\oD$ to a complex geodesic $\eta_0$ of $\B^n$ such that $\eta(\oD)\cap B(0,R)=\emptyset$ (here $B(0,R):=\{z\in \C^n: |z|<R\}$).

\smallskip

Let $a\in (0,1)$. Recall that $\varphi_k'(1)=(\alpha_k, v_k)$. Since $\alpha_k$ is strictly decreasing and converges to $0$, we can find a sequence $\{t_k\}$ increasing and converging to $1$ such that 
\[
\sqrt{\frac{1+t_k}{1-t_k}}\alpha_k=a.
\]
Pre-composing with a hyperbolic automorphism of $\D$ fixing $1$, we can parametrize $\varphi_k$ in such a way that $\varphi_k'(1)=\sqrt{\frac{1-t_k}{1+t_k}}(\alpha_k, v_k)$. Thus, $(A_{t_k}\circ \varphi_k)'(1)=(a, v_k)$. 

Arguing as in the proof of Claim B, we see that, if $a$ is sufficiently small so that the complex geodesic $\eta$ of $\B^n$ such that $\eta(1)=e_1$ and $\eta'(1)=(a,v)$ is contained in $\{\Re z_1\geq 0\}$ then
$\{A_{t_k}\circ \varphi_k\}$ converges (possibly up to a further reparametrization) in the $C^1$-topology on $\oD$ to $\eta$. Taking $a$ small (depending on $R$) we have the claim.
\smallskip

{\sl Claim C}. $\{F\circ \tilde\varphi_k\}$ converges (up to subsequences and reparametrization) to a complex geodesic of $G$.

\smallskip

Since  the Euclidean diameter of $\{F\circ \tilde\varphi_k\}$ is uniformly bounded from below from zero because $F(z_k)\to w_0\in G$, the claim follows by the same arguments as above.

\smallskip

{\sl Claim D}. Let $r_0>0$ be such that $[r_0,1)e_1\subset D\cap V$. Then there exists $M>0$ such that for every $t\in [r_0, 1)$ and $k\in \N$ there exists $\xi_{t,k}\in (0,1)$ such that 
\[
K_G(F(\Phi^{-1}(te_1)), F(\tilde\varphi_k(\xi_{t,k})))\leq M.
\]

\smallskip

Note that the curve $\Phi^{-1}(te_1)$ converges non-tangentially to $p$ as $t\to 1$. By Step 3, the curve $\gamma(t):=F(\Phi^{-1}(te_1))$, $t\in [r_0,1)$, converges to $p'$ as $t\to 1$. Thus, by Lemma~\ref{Lem:control-dist}, we have only to prove that there exists a complex geodesic $\varphi_0$ of $G$ such that $\varphi_0(1)=p'$ and $\gamma(t)$ and $\varphi_0$ satisfy  \eqref{Eq:control-nontg}. To this aim, let $f\in \mathcal F$ be such that $f(1)=p$. By hypothesis, $\varphi_0:=F\circ f$ is a complex geodesic in $G$ and $F(f(1))=p'$ (again, by Step 3). Since $F$ does not increase the Kobayashi distance, it is enough to show that there exists $C>0$ such that for every $t\in [r_0, 1)$
\[
K_D(\Phi^{-1}(te_1), f([0,1))\leq C.
\]
Since $\Phi^{-1}(te_1)$ converges non-tangentially to $p$, this follows at once from Remark~\ref{cono2}.

\smallskip

Now, let $M>0$ be as in Claim D.  Fix $R\in (0,1)$ so that $\overline{B_{\B^n}^K(0,2M)}\subset \{|z|\leq R\}$ (here $B_{\B^n}^K(0,M)$ denotes the Kobayashi ball in $\B^n$ of center $0$ and radius $M$). Let $\{t_k\}$ be given by Claim $B$ and let $\xi_k:=\xi_{t_k,k}\in (0,1)$ be given by Claim D.  Finally, let $w_k:=F(\Phi^{-1}(t_ke_1))$. 

As in the proof of Lemma~\ref{non-dic-close-bd}, we can find univalent maps $\Psi_k:G\to \Psi_k(G)=:G_k$   such that $r_k \B^n\subset G_k\subset \B^n$, $\Psi_k(w_k) =0$, where $r_k>0$ and $r_k\to 1$ as $n\to \infty$.

Let $D_k:=A_{t_k}(\Phi(D\cap V))$ and $F_k:=\Psi_k \circ F\circ \Phi^{-1}\circ A_{t_k}^{-1}:D_k \to G_k$.  Note that $F_k(0)=0$. Therefore, a Montel's argument implies that, up to subsequences, $\{F_n\}$ converges on compacta to a holomorphic map $H:\B^n\to \B^n$.

Arguing as in  Lemma~\ref{non-dic-close-bd} and by Claim A,  we see that $H$ acts as an isometry for the Kobayashi distance on every complex geodesic $\eta$ of $\B^n$ such that $\eta(\D)\subset \{\Re z_1>0\}$ and $\eta(1)=e_1$. Since clearly the closure of the union of  such complex geodesics contains an open subset of $\partial \B^n$, it follows from Theorem~\ref{main-ball} that $H$ is an automorphism of $\B^n$ and $H(0)=0$. In particular, $H(B_{\B^n}^K(0, T))=B_{\B^n}^K(0, T)$ for all $T>0$. 

By Claim $B$, the complex geodesics $\{A_{t_k}\circ \varphi_k\}$ converges (in the $C^1$-topology of $\oD$) to the complex geodesic $\eta_0$ of $\B^n$ which, by our choice of $R$, is not contained in $\overline{B_{\B^n}^K(0,2M)}$. On the other hand, by Claim $D$, the sequence of complex geodesics $\{F_k\circ A_{t_k}\circ \varphi_k\}$ of $G_k$ converges (up to extracting subsequences and reparametrizan) to a complex geodesic $\eta_1$ of $\B^n$ which intersects $B_{\B^n}^K(0,M)$ (since $K_{G_k}$ converges uniformly on compacta to $K_{\B^n}$). 

Now, let $z\in \eta_0(\D)$. We can find a sequence $z_k\in A_{t_k}(\varphi_k(\D))$ such that $z_k\to z$. Since $\{F_k\}$ converges to $H$ on compacta of $\B^n$, it follows that $H(z)=\lim_{k\to \infty}F_k(z_k)$. This implies also that $\{F_k(z_k)\}$ is contained in a compact subset $Q$ of $\B^n$. Moreover, by construction, $F_k(z_k)\in \Psi_k(F(\tilde\varphi_k(\D))$, and, by Claim  C, $\Psi_k(F(\tilde\varphi_k(\D))\cap Q$ converges (in the Hausdorff sense) to $\eta_1(\D)\cap Q$. Therefore, $H(z)\in \eta_1(\D)$, and, by the arbitrariness of $z$, we have  $H(\eta_0(\D))=\eta_1(\D)$. 

Summing up, we have proved that  there exists $z\in \eta_0(\D)\cap B_{\B^n}^K(0, M)$ such that $H(z)\in \eta_1(\D)\cap B_{\B^n}^K(0, M)=\emptyset$, a contradiction and Step 4 follows.

\medskip

{\sl Step 5.} If $p\in D$ then $F$ is a biholomorphism.
\smallskip

By Step 4, $F$ is proper and hence, according to \cite[Theorem 1]{DF},  $F$ is  a biholomorphism. 
\end{proof}

\end{document}